\documentclass[reqno,12pt]{amsart}
\usepackage{a4wide}
\usepackage{amsmath}
\usepackage{amssymb}
\usepackage{amsthm}
\usepackage{mathrsfs}
\usepackage{epsfig}
\usepackage[utf8]{inputenc} 

\numberwithin{equation}{section}

\newtheorem{theo}{Theorem}

\newtheorem{prop}{Proposition}

\theoremstyle{remark}
\newtheorem{Remark}{Remark}

\def\Qbar{\overline{\mathbb Q}}

\def\Z{\mathbb{Z}}
\def\Q{\mathbb{Q}}
\def\C{\mathbb{C}}

\def\({\left(}
\def\){\right)}
\def\[{\left[}
\def\]{\right]}

\newcommand{\calA}{\mathcal A}
\newcommand{\cali}{\mathcal I}
\newcommand{\Span}{{{\rm Span}}}
\newcommand{\Card}{{\rm Card} \,}
\newcommand{\unq}{\{1,\ldots,q\}}
\newcommand{\K}{\mathbb K}
\newcommand{\LL}{\mathbb L}
\newcommand{\asoul}{\underline{a}}

\begin{document}
\title[]
{A note on $G$-operators of order $2$}
\author[]{S. Fischler et T. Rivoal}

\begin{abstract} It is known that $G$-functions solutions of a linear differential equation of order 1 with coefficients in $\Qbar(z)$ are algebraic (of a very precise form). No general result is known when the order is $2$. 
In this paper, we determine the form of a $G$-function solution of an inhomogeneous equation of order 1 with coefficients in $\Qbar(z)$, as well as that of a $G$-function $f$ of differential order 2 over $\Qbar(z)$ and such that $f$ and $f'$ are algebraically dependent over $\mathbb C(z)$. Our results apply more generally to holonomic Nilsson-Gevrey arithmetic series of order 0 that encompass $G$-functions. 
\end{abstract}

\date{\today}

\address{
 S. Fischler, 
Universit\'e Paris-Saclay, CNRS, Laboratoire de math\'ematiques d'Orsay, 91405, Orsay, France.
\newline
\newline
\indent T. Rivoal, Institut Fourier, 
 CNRS et Universit\'e Grenoble Alpes, 
 CS 40700, 
 38058 Grenoble cedex 9, France.
}



\maketitle

\section{Introduction}

We fix an embedding of $\Qbar$ into $\mathbb C$. A $G$-function is a power series 
$$
f(z)=\sum_{n=0}^{\infty} a_n z^n \in \Qbar[[z]]
$$ 
such that:
\begin{enumerate} 
\item[--] $f(z)$ satisfies a non-zero linear differential equation with coefficients in $\Qbar(z)$; 
\item[--] there exists $C>0$ such that 
for any $\sigma \in \textup{Gal}(\overline{\mathbb{Q}}/\mathbb Q)$, we have $\vert \sigma(a_n)\vert \le C^{n+1}$;
\item[--]  there exists a sequence of positive integers $d_n$ such that $d_n \le C^{n+1}$ and $d_n a_m$ is an algebraic integer for all $m\le n$.
\end{enumerate}
This class of arithmetic power series was defined by Siegel~\cite{siegel}. Given some sub-field $\mathbb L$ of $\mathbb C$, throughout the paper, by 
``solution of a differential operator $L \in \mathbb L(z)[\frac d{dz}]$'', it must be understood ``solution of the homogeneous linear differential equation $Ly(z)=0$''. We say that a non-zero solution of a differential operator in $\mathbb L(z)[\frac d{dz}]$ of order $\mu$ is ``of order $\mu$ over $\mathbb{L}(z)$'' if it is not solution of a differential operator in $\mathbb L(z)[\frac d{dz}]\setminus\{0\}$ of order $\le \mu-1$. (Note that $y(z)=z+\pi$ is of order 1 over $\mathbb C(z)$ but of order 2 over $\Qbar(z)$, because $(z+\pi)y'(z)-y(z)=0$ and $y''(z)=0$.) Amongst differential operators in $\Qbar(z)[\frac{d}{dz}]$ of which $G$-functions can be solutions, we distinguish $G$-operators (precisely defined in \S\ref{sec:1}); in fact, any $G$-function $f(z)$ is always solution of a $G$-operator, and this highly non-trivial fact implies many special properties of the minimal linear differential equation satified by $f(z)$ over $\Qbar(z)$ (fuchsianity, rationality of exponents, etc). It is also conjectured that the class of $G$-operators and the class of globally nilpotent operators in $\Qbar(z)[\frac{d}{dz}]$ coincide (see \cite{andre} for more on this).

Standard examples of $G$-functions are algebraic functions over $\Qbar(z)$ regular at the origin, as well as the generalized hypergeometric series ${}_{p+1}F_p[a_1,\ldots, a_{p+1}; b_1, \ldots, b_p;z]$ with rational parameters $a_j$ and $b_j$ (algebraic parameters are also possible in certain circumstances). It is not known if there is a way to algebraically express any $G$-function in terms of these two classes of functions only, although this has been the subject of some speculations and conjectures. For instance, Dwork conjectured in \cite{dwork} that any globally nilpotent operator in $\Qbar(z)[\frac{d}{dz}]$ of order 2 (hence conjecturally any $G$-operator of order 2) is a linear differential equation satisfied by an algebraic function over $\Qbar(z)$ or an algebraic pullback of Gauss' hypergeometric differential equation satisfied by ${}_2F_1[a,b;c;z]$ for some $a,b,c\in \Qbar$~(\footnote{According to this conjecture, a solution of a globally nilpotent operator could then be expressed as a $\mathbb C$-linear combination of functions of the form $\mu(z)F[z\lambda(z)]$ where $\mu(z), \lambda(z)\in \Qbar(z)$ are regular at $z=0$, and $F(z)$ is a solution of Gauss' hypergeometric equation.}). Dwork's conjecture has been disproved by Krammer \cite{krammer}, further counter-examples being given later on in \cite{bouw}. On a related note, Theorem 5 in \cite[\S 7]{firipadoue} shows that it is very unlikely that any $G$-function could be written as a polynomial with coefficients in $\Qbar$ of $G$-functions of the form $\mu(z)\cdot{}_{p+1}F_p[a_1, \ldots, a_{p+1}; b_1,\ldots, b_p ;z\lambda(z)]$, with $p \ge 0$, $a_j\in \mathbb Q$, $b_j\in \mathbb Q\setminus \mathbb Z_{\le 0}$, $\mu(z), \lambda(z)$ algebraic over $\Qbar(z)$ and regular at $z=0$. On the other hand, a folklore assertion is that $G$-functions can be obtained as suitable specialisations of the parameters and variables of the multivariate series known as $A$-hypergeometric functions (defined by Gelfand-Kapranov-Zelivinski, see \cite{gkz} and references there). This is already known to be true for algebraic functions over $\Qbar(z)$, see \cite{sturmfels} for instance.

To state our results, we first present a few results of the solutions of $G$-operators. 
Andr\'e introduced in \cite{andre} the class of  Nilsson-Gevrey arithmetic series of order 0, denoted by $\textup{NGA}\{0\}_0^{\mathbb C}$ (see \S\ref{sec:holo} for the definition). In this paper we consider only {\em holonomic}  Nilsson-Gevrey arithmetic series of order 0, and we prove the following result. By holonomic, we mean ``solution of a linear differential equation over $\mathbb C(z)$''.

\begin{prop} \label{propholo} Holonomic  Nilsson-Gevrey arithmetic series of order 0 are exactly the   functions of the form
\begin{equation}\label{eq:4}
\sum_{(\alpha,j,k)\in \mathbb S} \lambda_{\alpha,j,k} z^{\alpha} \log(z)^j f_{\alpha,j,k}(z)
\end{equation}
where  $\mathbb S$ is a finite subset of $\mathbb Q\times \mathbb N\times \mathbb N$, $\lambda_{\alpha,j,k}\in \mathbb C$ and each $f_{\alpha,j,k}(z)$ is a   $G$-function.
\end{prop}

In this definition, the determination of $\log(z)$ is arbitrary.  This notion could be  also considered at another point of $\mathbb C\cup\{\infty\}$ with obvious changes. It is known that any solution of a $G$-operator is in $\textup{NGA}\{0\}_0^{\mathbb C}$ (see \S\ref{sec:1}) and conversely Andr\'e proved that any holonomic element of $\textup{NGA}\{0\}_0^{\mathbb C}$ is solution of a $G$-operator (see \cite[p. 720]{andre}). We shall also consider the sub-class $\textup{NGA}\{0\}_0^{\Qbar}$ of $\textup{NGA}\{0\}_0^{\mathbb{C}}$ where
 the coefficients $\lambda_{\alpha,j,k}$ in Eq. \eqref{eq:4} are in $\Qbar$. Algebraic functions over $\mathbb{C}(z)$, respectively over $\Qbar(z)$, are in $\textup{NGA}\{0\}_0^{\mathbb{C}}$, respectively in $\textup{NGA}\{0\}_0^{\Qbar}$.

\bigskip

The goal of this paper is to  describe the holonomic elements of $\textup{NGA}\{0\}_0^{\mathbb C}$ of order 2 over $\Qbar(z)$ and subject to certain restrictions (Proposition \ref{coro:1} and Theorem \ref{theo:2} below). They are in agreement with the ``conjecture'' recalled above that $G$-functions should be specialisations of $A$-hypergeometric functions. 
The structure of the general  holonomic elements of $\textup{NGA}\{0\}_0^{\mathbb C}$ of order 2 over $\Qbar(z)$ remains unknown. Throughout the paper, $\int u(z) dz$ denotes a primitive of a function $u(z)$ where the arbitrary complex constant is not specified, while the definite integral $\int_{z_0}^z u(x) dx$ denotes the primitive of $u(z)$ that vanishes at $z=z_0$.

\bigskip

To begin with, let us consider the easy case of solutions of homogeneous operators of degree 1.
We explain in \S\ref{sec:1} why the non-zero elements of $\textup{NGA}\{0\}_0^{\mathbb C}$ solutions of an operator in $\Qbar(z)[\frac{d}{dz}]$ of order 1 are exactly the functions of the form
\begin{equation} \label{eq:1}
\delta\prod_{j\in J} (\lambda_j-z)^{s_j} 
\end{equation}
where $\delta\in \mathbb C$, $J$ is a finite set, $\lambda_j\in \Qbar$ and $s_j\in \mathbb Q$ for every $j\in J$. In particular, up to a multiplicative constant, these solutions are algebraic functions over $\Qbar(z)$.

\bigskip

Now we move to solutions of inhomogeneous operators of degree 1 (see \cite[pp. 398--399]{Zannier} for similar considerations in a particular case).
\begin{prop} \label{coro:1} Let $f(z)\notin \Qbar(z)$ be an element of $\textup{NGA}\{0\}_0^{\mathbb C}$ and solution of an inhomogeneous differential equation $f'(z)=a(z)f(z)+b(z)$ with $a(z)\in \Qbar(z)$ and $b(z)\in \Qbar(z)^*$. Then, $\frac{d}{dz}-a(z)$ is a $G$-operator and letting $g(z)\neq 0$ be one of its solutions, we have 
\begin{equation} \label{eq:3}
f(z)=g(z)\int \frac{b(z)}{g(z)}dz.
\end{equation}
\end{prop}
\begin{Remark} 
The assumptions of $f(z)$ ensure that it is of order 2 over $\Qbar(z)$ 
(see the details in the proof).

The important fact in Proposition \ref{coro:1} is that $\frac{d}{dz}-a(z)$ is a $G$-operator, which was not obvious {\em a priori}.

A partial converse of Proposition \ref{coro:1} holds. With $g(z)\neq 0$ a solution of a $G$-operator $\frac{d}{dz}-a(z)$ and $b(z)\in \Qbar(z)^*$, the right-hand side of \eqref{eq:3} is solution of the inhomogeneous equation $y'(z)=a(z)y(z)+b(z)$, and its generalized expansion at $z=0$ is in $\textup{NGA}\{0\}_0^{\mathbb C}$. But it is not guaranteed that it is not in $\Qbar(z)$.

The assumption $b(z)\neq 0$ is important for the proof but not really restrictive because \eqref{eq:1} gives the form of the elements of $\textup{NGA}\{0\}_0^{\mathbb C}$ solutions of an operator $\frac{d}{dz}-a(z)\in \Qbar(z)[\frac d{dz}]$. 
\end{Remark}

\bigskip

Let us state now our main result, which deals with solutions of homogeneous operators of degree 2 (with an additional assumption, namely that $f$ and $f'$ are algebraically dependent over $\mathbb C(z)$).

\begin{theo} \label{theo:2} Let $f(z)\neq 0$ be an element of $\textup{NGA}\{0\}_0^{\mathbb C}$, holonomic  
of order $
2$ over $\Qbar(z)$ 
and such that $f(z)$ and $f'(z)$ are algebraically dependent over $\mathbb C(z)$. Let $L\in \Qbar(z)[\frac{d}{dz}]$ of order 2 be such that $Lf(z)=0$. Then, at least one of the following assertions holds:

$(i)$ $f(z)$ is algebraic over $\mathbb C(z)$. More precisely, the differential equation $Ly(z)=0$ has a basis of solutions made of algebraic functions over $\Qbar(z)$.

$(ii)$ There exist two solutions $g(z)\neq 0, h(z)\neq 0$ of (possibly distinct) $G$-operators of order 1 
such that
\begin{equation} \label{eq:2b}
f(z)=g(z)\int h(z) dz.
\end{equation}
The functions $f(z)$ and $g(z)$ form a basis of solutions of the differential equation $Ly(z)=0$.
\end{theo}

\begin{Remark}
Assertions $(i)$ and $(ii)$ can hold simultaneously. The function $f(z)=\pi\sqrt{z}-e$ (with its principal branch) is of order 2 over $\Qbar(z)$ with $L=2z\frac{d^2}{dz^2}+\frac{d}{dz}$ (and 1, $\sqrt{z}$ as a basis), and we can take $g(z)=1$ and $h(z)=\pi/(2\sqrt{z})$ in $(ii)$. On the other hand, a function $f(z)$ as in \eqref{eq:2b} can be transcendental over $\mathbb C(z)$ with $f(z), f'(z)$ algebraically dependent over $\mathbb{C}(z)$: take $g(z)=1$ and $h(z)=1/z$ for instance.

A partial converse of Theorem \ref{theo:2} holds. Concerning $(i)$: any function $f(z)$ algebraic over $\mathbb C(z)$ is in $\textup{NGA}\{0\}_0^{\mathbb C}$ and obviously such that $f(z)$ and $f'(z)$ are algebraically dependent over $\mathbb C(z)$. Concerning $(ii)$: let $f(z)$ be as in \eqref{eq:2b} with $g(z)\neq 0, h(z)\neq 0$ solutions of $G$-operators of order 1 (hence both algebraic over $\mathbb C(z)$). Then 
 its generalized expansion at $z=0$ is in $\textup{NGA}\{0\}_0^{\mathbb C}$. Moreover $f(z)$ satisfies $f'(z)-a(z)f(z)=h(z)g(z)$ where $a(z)\in \Qbar(z)$ (such that $g'(z)=a(z)g(z)$) and $h(z)g(z)$ is algebraic over $\mathbb C(z)$ and of order~1 over $\Qbar(z)$. Hence, $f(z)$ and $f'(z)$ are algebraically dependent over $\mathbb C(z)$, and $f(z)$ is of order $\le 2$ over $\Qbar(z)$. 

Under the assumptions of Theorem \ref{theo:2}, $f(z)$ turns out to be a Liouvillian solution of the $G$-operator $L$. Conversely, the proof of the theorem (based on Kovacic's classification) shows that any Liouvillian solution of a $G$-operator $L$ of order 2 is either algebraic over $\mathbb C(z)$ or of the form \eqref{eq:2b}.

It is not easy to state simple necessary and sufficient conditions ensuring that the right-hand side of \eqref{eq:2b} is a $G$-function, not merely 
a holonomic  element of $\textup{NGA}\{0\}_0^{\mathbb C}$, because there are many possible situations. We can write $g(z)=\delta z^{\alpha}\widetilde{g}(z)$ and $h(z)=\omega z^{\beta}\widetilde{h}(z)$ where $\delta, \omega\in \mathbb C^*$, $\alpha, \beta\in \mathbb Q$ and $\widetilde{g}(z), \widetilde{h}(z)$ are $G$-functions such that $\widetilde{g}(0)\widetilde{h}(0)\neq 0$. Let $\widetilde{h}(z)=\sum_{n=0}^\infty a_n z^n$. If we assume for instance that $\beta>-1$, then a necessary and sufficient condition for $g(z)\int_0^z h(x) dx$ to be a $G$-function is that $\alpha+\beta\in \mathbb Z_{\ge -1}$ and $\delta\omega\in \Qbar$. Indeed, we have $g(z)\int_0^z h(x) dx=\delta\omega z^{\alpha+\beta+1} \widetilde{g}(z)\sum_{n=0}^\infty \frac{a_n z^n}{\beta+n+1}$ with $\frac{\widetilde{g}(0)a_0}{\beta+1}\neq 0$.
\end{Remark}

\medskip

To conclude, we mention that similar questions have already been addressed for $E$-functions in Siegel's original sense. $E$-functions of order 1 over $\Qbar(z)$ have been determined by Shidlovskii \cite[p.~184]{shid}. Building upon a remark in \cite[p.~724, \S4.5]{andre}, $E$-functions solutions of inhomogeneous differential equations of order 1 over $\Qbar(z)$ have been classified by Gorelov \cite{gorelov1}, a result reproved in \cite{rivroq} for $E$-functions in the restricted sense. Gorelov eventually classified $E$-functions of order 2 over $\Qbar(z)$ in \cite{gorelov2}, and a different proof was also given in \cite{rivroq2} for $E$-functions in the restricted sense. This classification involves only ${}_1F_1$ hypergeometric series with rational parameters. We also emphasize that the above Theorem \ref{theo:2} is an analogue of Theorem 3 of \cite{rivroq}, and we drew inspiration of the proof of the latter (based on Kovacic's theorem \cite{kovacic} adapted to $\Qbar(z)$) to prove the former; the main difference is that in the proof of Theorem \ref{theo:2} in \S\ref{sec:2} below, Cases 1, 2 and 3 can happen in Kovacic's classification, while in the proof of \cite[Theorem 3]{rivroq} only Case 1 can happen. It does not seem that the methods of \cite{gorelov2} or \cite{rivroq2} can be easily adapted to classify $G$-functions of order 2 over $\Qbar(z)$.

\medskip

The structure of this paper is as follows. In \S \ref{sec:1} we recall the results we shall use on $G$-operators, and study the solutions in $\textup{NGA}\{0\}_0^{\mathbb C}$ of $G$-operators of order 1, and those of $G$-operators of order 2 reducible over $\Qbar(z)$. 
We  prove Proposition \ref{coro:1}, and also  that any $L\in \Qbar(z)[\frac d{dz}]$ is of minimal order over $\Qbar(z)$ for one of its solutions (viewed as an element of a Picard-Vessiot extension of $Ly(z)=0$ over $\mathbb C(z)$) -- a result that does not hold with $\Qbar(z)$ replaced by $\mathbb{C}(z)$. 
In \S \ref{sec:2} we prove   Theorem \ref{theo:2}. At last in \S \ref{sec:holo} we prove 
Proposition \ref{propholo}, and 
 in \S \ref{sec:3} an independent result that we have not found in the literature: any function algebraic over $\C(z)$ and holonomic over $\Qbar(z)$ is a $\C$-linear combination of functions algebraic over $\Qbar(z)$.

\medskip

\noindent {\bf Acknowledgements.} We thank C. Hardouin, J. Roques and J.-A. Weil for patiently answering our questions on differential Galois theory (in particular for pointing a crucial result in \cite{katz87}) and Y. Andr\'e for confirming to us that a result of his on $G$-functions (scholie in \cite[p.~123]{andrelivre}) can be extended to holonomic elements of $\textup{NGA}\{0\}_0^{\Qbar}$ with the same proof. We also thank M. Singer for his comments on our Proposition \ref{propalg} in the final part of the paper, and for sending us his own alternative proof of it (not reproduced here). Finally, let us mention that the starting point of this paper was a question of the referee of our paper \cite{firipadoue}: he asked us if something could be said of $G$-functions solutions of inhomogeneous equations of order~1 over $\Qbar(z)$. Both authors have partially been funded by the ANR project De Rerum Natura (ANR-19-CE40-0018).

\section{Some results on $G$-operators} \label{sec:1}

Consider a differential system $Y'(z)=A(z)Y(z)$ with $A(z)\in M_{s\times s}(\Qbar(z))$. It is immediate that $Y^{(n)}=A_n(z)Y(z)$ where the sequence of matrices $(A_n)_{n\ge 1}$ is defined by $A_{n+1}=A_nA_1+A_n'$, $A_1:=A$. 

Let $T(z)\in \Qbar[z]\setminus\{0\}$ (of minimal degree) such that $T(z)A(z)$ has entries in $\Qbar[z]$. It is easy to check by induction that, for every $n$, $T(z)^n A_n(z)$ has entries in $\Qbar[z]$. Let $D_k\ge 1$ denote the least integer such that $\frac{D_k T(z)^n}{n!}A_n(z)$, $n=1, \ldots, k$, all have entries in $\mathcal{O}_{\Qbar}[z]$. We say that $A(z)$ satisfies Galochkin's condition if $D_k$ has at most geometric growth (see \cite{galochkin}). We say that the differential system $Y'(z)=A(z)Y(z)$ is a $G$-operator when $A(z)$ satisfies Galochkin's condition. By extension, a differential operator in $\mathbb C(z)[\frac d{dz}]$ is said to be a $G$-operator when its companion differential system is a $G$-operator; in particular there exists $p(z)\in \mathbb C(z)$ such that $p(z)L\in \Qbar(z)[\frac d{dz}]$ and there is no loss of generality in considering that $G$-operators are in $\Qbar(z)[\frac d{dz}]$.

If $L_1, L_2\in \Qbar(z)[\frac d{dz}]$ are $G$-operators, then $L_1L_2$ is a $G$-operator. Conversely, if $L\in \Qbar(z)[\frac d{dz}]$ is a $G$-operator that can be factorized as $L=L_1L_2$ with $L_1, L_2\in \Qbar(z)[\frac d{dz}]$, then $L_1$ and $L_2$ are $G$-operators. See \cite{andrelivre} or \cite[p.~16, Corollary~2]{lepetit} for a proof.

Andr\'e \cite{andrelivre} proved that Galochkin's condition is equivalent to another one introduced by Bombieri \cite{bombieri}. Bombieri's condition and a result of Katz imply that a $G$-operator is fuchsian with rational exponents. Moreover, various estimates from the theory of $p$-adic differential equations imply that at any point of $\Qbar\cup \{\infty\}$, a $G$-operator has a local basis of solutions (essentially) made of $G$-functions. In particular and more precisely, the local solutions at $z=0$ of a $G$-operator are in $\textup{NGA}\{0\}_0^{\mathbb C}$. The converse is true by a theorem of Andr\'e quoted below.

It is difficult to prove that a differential operator is $G$-operator because Galochkin's condition can be hard to verify. Chudnovsky \cite{chud} proved the following sufficient condition: if $L\in \Qbar(z)[\frac{d}{dz}]\setminus \{0\}$ is of minimal order over $\Qbar(z)$ for some $G$-function, then $L$ is a $G$-operator. 
In \cite[p.720]{andre}, Andr\'e extended Chudnovsky's theorem: any holonomic element of $\textup{NGA}\{0\}_0^{\mathbb C}$ is solution of a $G$-operator. Therefore, if $L\in \Qbar(z)[\frac{d}{dz}]\setminus \{0\}$ is of minimal order over $\Qbar(z)$ for some element in $\textup{NGA}\{0\}_0^{\mathbb C}$, it is a $G$-operator. We will mention this result as {\em Andr\'e's minimality theorem} in the rest of the paper.

A complete characterization of $G$-operators of order $\ge 2$ is not known, but this can be done when the order is $1$. 
\begin{prop}\label{prop:1} 
$(i)$ If a non-zero element $y(z)$ of $\textup{NGA}\{0\}_0^{\mathbb C}$ is a solution of $L\in \Qbar(z)[\frac d{dz}]$ of order~1, then $L$ is a $G$-operator and we have 
$y(z) = 
\delta\prod_{j\in J} (\lambda_j-z)^{s_j} 
$
where $\delta\in \mathbb C^*$, $J$ is a finite set, $\lambda_j\in \Qbar$ and $s_j\in \mathbb Q$ for every $j\in J$. 

$(ii)$ $G$-operators of order 1 are exactly the differential operators in $\Qbar(z)[\frac{d}{dz}]$ of order 1 which are fuchsian and with rational exponents. 
\end{prop}

\begin{proof} First of all, if $y(z)$ is a non-zero element of $\textup{NGA}\{0\}_0^{\mathbb C}$ and $L\in \Qbar(z)[\frac{d}{dz}]$ is a differential operator of order 1 such that $Ly(z)=0$, then 
 $L$ is clearly of minimal order over $\Qbar(z)$ for $y(z)$. 
 By Andr\'e's minimality theorem (see also Remark \ref{rem3} below), $L$ is a $G$-operator.

\medskip

Now consider $L\in \Qbar(z)[\frac{d}{dz}]$ of order 1 and fuchsian with rational exponents. Without loss of generality, we assume that $L$ is monic, ie that 
$$
L=\frac{d}{dz}-\sum_{j\in J} \frac{s_j}{\lambda_j-z} 
$$
with $\lambda_j\in \Qbar$ (pairwise distinct) and $s_j\in \mathbb Q$ for all $j\in J$
(see \cite[Lemma 6.11, p. 174]{VDPSinger}).
 Hence, the solutions of $L$ are of the form
\begin{equation*} 
\delta\prod_{j\in J} (\lambda_j-z)^{s_j} \in \textup{NGA}\{0\}_0^{\mathbb C}
\end{equation*}
where $\delta\in \mathbb C$, $J$ is a finite set, $\lambda_j\in \Qbar$ and $s_j\in \mathbb Q$ for every $j\in J$. (If $J=\emptyset$, the value of the product is $1$.) Since $L$ is minimal over $\Qbar(z)$ for the non-zero solution with $\delta=1$, as above we deduce that $L$ is a $G$-operator. Since every $G$-operator is fuchsian with rational exponents, this concludes the proof.
\end{proof}

\begin{Remark}\label{rem3}
Andr\'e's minimality theorem is a general result for differential operators of arbitrary orders, and it can be avoided in this particular situation. Let us prove directly that Galochkin's condition holds for the operator $\frac{d}{dz}-A_1(z)$ with 
$
A_1(z):=\sum_{j\in J} \frac{s_j}{\lambda_j-z}.
$
Recall that we define a sequence of matrices $(A_n)_{n\ge 1}$ by $y^{(n)}(z)=A_n(z)y(z)$ where $y(z)$ is any solution of $\frac{d}{dz}-A_1(z)$ (the sequence is independent of the solution). Let $J=\{1, \ldots, p \}$. Taking 
$y(z)=\prod_{j=1}^p (\lambda_j-z)^{s_j}$, Leibniz's formula shows 
that 
$$
A_n(z)= n! \sum_{n_1+\cdots +n_p=n} \Big(\prod_{j=1}^p \frac{(-s_j)_{n_j}}{n_j!}\frac{1}{(\lambda_j-z)^{n_j}}\Big).
$$
Since for any $t\in \mathbb Q$, the common denominator of the numbers $\frac{(t)_n}{n!}$, $n=0, \ldots, k$, has at most geometric growth in $k$ (Siegel), it follows that Galochkin's condition is satisfied by the matrices $\frac{T(z)^n}{n!}A_n(z)$, $n=1, \ldots, k$, with $T(z)=\prod_{j=1}^p (\lambda_j-z)$.
\end{Remark}

\begin{prop} \label{theo:1} Let $f(z)\in \textup{NGA}\{0\}_0^{\mathbb C}$ be a non-zero solution of a $G$-operator $L\in \Qbar(z)[\frac d{dz}]$ of order $2$ which is reducible over $\Qbar(z)$. Then, for any factorization $L=MN$ with $M,N \in \Qbar(z)[\frac d{dz}]$ of order 1,

$(i)$ $M$ and $N$ are both G-operators,

$(ii)$ there exist a solution $g(z)\neq 0$ of $N$ (and thus of $L$) and a solution $k(z)$ of $M$ 
such that
\begin{equation} \label{eq:2}
f(z)=g(z)\int \frac{k(z)}{g(z)} dz.
\end{equation}
\end{prop}

\begin{proof}
Since $L$ of order $2$ is reducible over $\Qbar(z)$, there exist $\eta(z), \varpi(z), \nu(z)\in\Qbar(z)$, $\nu(z)\neq 0$, such that
$$
L=\nu(z)\Big(\frac{d}{dz}-\eta(z)\Big)\Big(\frac{d}{dz}-\varpi(z)\Big).
$$
$L$ being a $G$-operator, this is also the case of $M:=\nu(z)(\frac{d}{dz}-\eta(z))$ and $N:=\frac{d}{dz}-\varpi(z)$. 

We set $k(z):=f'(z)-\varpi(z)f(z)$, so that $Mk(z)=0$.~(\footnote{Notice that if $f(z)$ is a $G$-function, both $f(z)$ and $f'(z)$ are regular at $z=0$, and $\varpi(z)\in \Qbar(z)$, so that if $k(z)$ has a singularity at $z=0$, then it is a pole.})

We now have to solve for $f(z)$ the inhomogeneous equation $f'=\varpi f+k$. This is a well-known exercise. Let $g(z)$ be a non-zero solution of the homogeneous equation $y'=\varpi y$. Then, the general solution of $y'=\varpi y+k$ is of the form
$$
y(z)=g(z)\int \frac{k(z)}{g(z)} dz.
$$
In particular, $f(z)$ is of this form and the proof is complete. 
\end{proof}
\begin{Remark}\label{rem4}
 The converse of Proposition \ref{theo:1} holds: if $g$, $k$ are $G$-functions such that $Ng(z)=Mk(z)=0$, where $M,N \in \Qbar(z)[\frac d{dz}]$ are of order 1, then we may assume that $N = \frac{d}{dz}-a$ and $M=\frac{d}{dz}-b$, and any function $f$ defined by \eqref{eq:2b} 
satisfies $f'(z)-a(z)f(z)=k(z)$ so that $MNf(z)=0$, and its generalized expansion at $z=0$ is in $\textup{NGA}\{0\}_0^{\mathbb C}$. 

\smallskip

The operators $M$ and $N$ are not unique (even in the class of $G$-operators), as the classical factorization $\frac{d^2}{dz^2}=(\frac{d}{dz}-\frac{1}{z+\lambda})(\frac{d}{dz}+\frac{1}{z+\lambda})$, $\lambda\in \mathbb C$ arbitrary, shows; all these operators are $G$-operators when $\lambda\in \Qbar$. Hence the representation of $f(z)$ as in \eqref{eq:2} is not unique.
\end{Remark}

\bigskip

We are now able, using Proposition  \ref{theo:1}, to deduce   Proposition \ref{coro:1}.

We have $f'(z)=a(z)f(z)+b(z)$. Since $b(z)\neq 0$, $f(z)$ is then trivially a solution of the (reducible) differential operator
$$
L:=\Big( \frac{d}{dz}-\frac{b'(z)}{b(z)}\Big)\Big(\frac{d}{dz}-a(z) \Big) \in \Qbar(z)\Big[\frac{d}{dz}\Big]\setminus \{0\}.
$$
Because $f(z)\neq 0$, $f(z)$ is not of order 0 over $\Qbar(z)$ and because $b(z)\neq 0$ and $f(z)\notin \Qbar(z)$, $f(z)$ is not of order 1 over $\Qbar(z)$ either. In the latter case, assume on the contrary that $f'(z)=c(z)f(z)$ for some $c(z)\in \Qbar(z)$: we then have $a(z)f(z)+b(z)=c(z)f(z)$, so that either $f(z)\in \Qbar(z)$ or $b(z)=0$, which is impossible. Hence, $Lf(z)=0$ is of minimal order for $f(z)$ over $\Qbar(z)$. By Andr\'e's minimality theorem, it follows that $L$ is a $G$-operator, of which $f(z) \in \textup{NGA}\{0\}_0^{\mathbb C}$ is a solution. We are thus in the situation of Proposition \ref{theo:1}. More precisely, in the proof of that proposition, we may take $\varpi(z):=a(z)$, $\eta(z):=b'(z)/b(z)$, and $\nu(z):=1$; the function $k(z)$ defined there by $k(z):=f'(z)-\varpi(z)f(z)$ is thus equal to $b(z)$. This concludes the proof of Proposition \ref{coro:1}.

\bigskip

We end this section with a simple but interesting result which can be viewed as a converse to Andr\'e's minimality theorem when applied to a $G$-operator. We recall that $\Qbar$ is viewed as a sub-field of $\mathbb C$.

\begin{prop}\label{prop3} Any given $L\in \Qbar(z)[\frac{d}{dz}]$ admits a solution for which $L$ is of minimal order over $\Qbar(z)$.
\end{prop}
\begin{Remark} This is false if $\Qbar(z)$ is replaced by $\mathbb C(z)$. Consider for instance again $L=\frac{d^2}{dz^2}\in\mathbb C(z)[\frac{d}{dz}]$. The solutions of $L$ are $az+b$, $a,b\in \mathbb C$, and any particular function $az+b$ is solution of $(az+b)\frac{d}{dz}-a \in \mathbb C(z)[\frac{d}{dz}]$. This is of course another point of view on the above mentioned factorization of $\frac{d^2}{dz^2}$.
\end{Remark}
\begin{proof} We can assume that $L\neq 0$; we let $m\ge 1$ be the order of $L$. Let $\alpha\in \Qbar$ be an ordinary point of $L$ and let $f_1(z), \ldots, f_m(z)$ denote a local $\mathbb C$-basis of solutions of $L$ at $z=\alpha$ where each $f_j(z)$ is in $\Qbar[[z-\alpha]]$. Let $\omega_1, \ldots,\omega_m$ be $m$ complex numbers linearly independent over $\Qbar$ (for instance $\omega_j=\pi^j$) and consider the solution $h(z)=\sum_{j=1}^m \omega_jf_j(z)\neq 0$ of $L$. Let $M\in \Qbar(z)[\frac{d}{dz}]$ be of minimal order for $h(z)$. We have 
\begin{equation}\label{eq:omeMf}
Mh=\sum_{j=1}^m\omega_j Mf_j=0.
\end{equation}
Consider the Laurent series expansions $Mf_j(z)=\sum_{k=-K}^{\infty} \phi_{k,j}(z-\alpha)^k$, $j=1, \ldots, m$, $\phi_{k,j}\in \Qbar$ (where $K\ge 0$ can be chosen the same for all $j$). From \eqref{eq:omeMf}, we deduce the relations
$$
\sum_{j=1}^m \phi_{k,j}\omega_{j}=0, \quad \forall k\ge -K.
$$ 
The assumption on the $\omega_j$'s implies that $\phi_{k,j}=0$ for all $k\ge -K$ and all $j\in\{1,\ldots,m\}$. In other words, $Mf_j(z)=0$ for every $j$. By $\mathbb C$-linear independence of the $f_j$'s, we deduce that the order of $M$ is $m$. Thus $M=L$, up to multiplication by a non-zero element of $\Qbar(z)$.
\end{proof}

\section{Proof of Theorem \ref{theo:2}} \label{sec:2}

Let $f(z)\neq 0$ be an element of $\textup{NGA}\{0\}_0^{\mathbb C}$,  
holonomic of order $2$ over $\Qbar(z)$. 
Let $L\in \Qbar(z)[\frac d {dz}]$ of order~2 be such that 
$Lf(z)=0$: it is a $G$-operator by Andr\'e's minimality theorem. 
The differential Galois groups of $L$ over $\mathbb C(z)$ and $\Qbar(z)$ respectively are ``equal'' in the sense that they are defined by the same algebraic relations, ie the former is obtained from the latter by extension of scalars from $\Qbar$ to $\mathbb{C}$; see Proposition 1.3.2 in \cite[p. 19]{katz87} (due to Gabber) for the precise statement of this fact. Moreover, assuming that $f(z)$ and $f'(z)$ are algebraically dependent over $\mathbb C(z)$, the differential Galois group of $L$ over $\mathbb C(z)$ does not contain $\textup{SL}_2(\mathbb C)$. Therefore the differential Galois group of $L$ over $\Qbar(z)$ does not contain $\textup{SL}_2(\Qbar)$. Kovacic's classification \cite[\S 1.2, Theorem]{kovacic} (adapted to the case where the field of constants is $\Qbar$; see \cite{ulmerweil}) then implies that one of the following cases holds:

\medskip

Case 1: $L$ has a non-zero solution $g(z)$ such that $g'(z)=a(z)g(z)$ for some $a(z)\in \Qbar(z)$.

\medskip

Case 2: $L$ has a basis of solutions $g(z)k(z)$ and $h(z)k(z)$ such that $g(z)$ and $h(z)$ are algebraic over $\Qbar(z)$ and $k'(z)=a(z)k(z)$ for some $a(z)\in \Qbar(z)$.

\medskip

Case 3: $L$ has a basis of solutions $g(z)$ and $h(z)$ such that $g'(z)=a(z)g(z)$ and $h'(z)=b(z)h(z)$ where $a(z)$ and $b(z)$ are distinct quadratic functions, solutions of the same quadratic equation over $\Qbar(z)$.

\medskip

Before going on, we give some details on bases of solutions. In Cases 2 and 3, Kovacic's analysis provides {\em a priori} $\Qbar$-bases of solutions with the stated properties. We have to explain why these solutions also form $\mathbb C$-bases. This is in fact an immediate application of Corollary 1.13 in \cite[p. 10]{VDPSinger} but let us explain this in our situation. By Lemma 1.12 in \cite[p. 9]{VDPSinger}, 
two elements $u,v$ of a differential field with field of constants $\mathcal{C}$ (of characteristic 0) are linearly independent over $\mathcal{C}$ if and only if their wronskian $uv'-u'v$ is non-zero. Consider now a differential operator $M\in \Qbar(z)[\frac{d}{dz}]$ of order 2 with a $\Qbar$-basis of solutions $f,g$ in a Picard-Vessiot extension of $My(z)=0$ over $\Qbar(z)$. Notice that $f$ and $g$ are also elements of a Picard-Vessiot extension of $My(z)=0$ over $\mathbb C(z)$; since their wronskian is non-zero, $f$ and $g$ remain linearly independent over $\mathbb C$. Hence, $f$ and $g$ also form a $\mathbb C$-basis of $M$. This explains why in Cases 2 and 3, the bases of solutions of $L$ are also bases over $\mathbb C$, a fact that will be used in the discussions below.

\subsection{Discussion of Case 1}

The differential operator $N:=\frac{d}{dz}-a(z)\in \Qbar(z)[\frac d{dz}]$ is a right-factor of the $G$-operator $L\in \Qbar(z)[\frac d{dz}]$. Hence $L$ is reducible over $\Qbar(z)$ and we have $L=MN$ with $M,N\in \Qbar(z)[\frac d{dz}]$ both of order~1. We can thus apply Proposition~\ref{theo:1}: we have that $f(z)=g(z)\int k(z)/g(z) dz$ where $g(z)\neq 0$ and $k(z)$ are solutions of the $G$-operators $N$ and $M$ respectively. Now, the expressions of $g(z)$ and $k(z)$ given by Proposition~\ref{prop:1} show that $h(z):=k(z)/g(z)$ is also a solution of a $G$-operator of order 1. Moreover, the functions $f(z)$ and $g(z)$ are linearly independent over $\mathbb C$. Indeed, if on the contrary $c_1f(z)+c_2g(z)=0$ for some $c_1, c_2\in \mathbb C$ not both 0, then necessarily $c_1\neq 0$ and thus $f(z)$ would be of order $\le 1$ over $\Qbar(z)$ which is excluded. This also implies that $h(z)$ can not be zero. We are thus exactly in situation $(ii)$ of Theorem~\ref{theo:2}.

\subsection{Discussion of Case 2} We work in a suitable simply connected cut plane where all functions under consideration are analytic. The function $G:=gk$ is a solution of $L$, hence it is in $\textup{NGA}\{0\}_0^{\mathbb C}$. It follows that $k=G/g$ has moderate growth at its finite singularities and $\infty$. Hence, $\frac{d}{dz}-a(z)\in \Qbar(z)[\frac d{dz}]$ is fuchsian by Fuchs' criterion (see \cite[p. 55]{poole}), and we denote its finite singularities by $\lambda_j$, $j \in J$; they are all in $\Qbar$. We have 
$$
k(z)=\delta \prod_{j\in J} (\lambda_j-z)^{s_j}
$$
with $\delta\in \mathbb C^*$ and $s_j\in \mathbb C$ for all $j$. We now prove that $s_j\in \mathbb Q$. Indeed, in the local generalized expansion (around any $\beta \in \mathbb C$)
$$
\frac{G(z)}{g(z)}=\sum_{(\alpha,j,k)\in \mathbb S} \lambda_{\alpha,j,k} (z-\beta)^{\alpha} \log(z-\beta)^j G_{\alpha,j,k}(z-\beta) 
$$
obtained from those of $G$ and $g$, with $G_{\alpha,j,k}(z-\beta)\in \mathbb C[[z-\beta]]$, the exponents $\alpha$ are in $\mathbb Q$ because $g$ is an algebraic function and $G$ is a solution of the $G$-operator $L$ with rational exponents. Therefore $s_j\in\Q$ for any $j$, and $\delta^{-1}k(z)$ is algebraic over $\mathbb \Qbar(z)$.

In conclusion, $\delta^{-1}g(z)k(z)$ and $\delta^{-1}h(z)k(z)$ are algebraic functions over $\Qbar(z)$ and they form a basis of $L$, in accordance with $(i)$ of Theorem \ref{theo:2}.

\subsection{Discussion of Case 3} This case is more complicated. 
We work in a suitable simply connected cut plane where all functions under consideration are analytic, and fix an arbitrary determination of the square root function. The quadraticity assumption on $a$ and $b$ ensures the existence of $r,s\in \Qbar(z)$ such that $\sqrt{s}\notin \Qbar(z)$, $a=r+\sqrt{s}$ and $b=r-\sqrt{s}$. Since $g$ and $h$ are solutions of the $G$-operator $L$, they are in $\textup{NGA}\{0\}_0^{\mathbb{C}}$. From now on, we limit our discussion to the case of $g$ because the results can be transfered immediately to~$h$. 

From the equation $g'=(r+\sqrt{s})g$, we deduce that 
$$
g(z)=c\exp\Big(\int r(z) dz\Big) \exp\Big(\int \sqrt{s(z)} dz\Big)
$$
for some $c\in \mathbb C^*$, where $\int$ denotes arbitrary but fixed primitives of the functions involved. 
The value of the constant $c\neq 0$ is in fact arbitrary because we can of course replace $g$ by any of its non-zero constant multiples in the above discussion. We now explain how to assign a specific value to $c$ that will suit our goals. 
Notice that the function 
$\widetilde{g}(z):=\exp(-\int r(z) dz)g(z)=c\exp(\int \sqrt{s(z)} dz)$ satisfies $\widetilde{g}'=\sqrt{s}\widetilde{g}$ and that the local expansion at $z=0$ of $\int \sqrt{s(z)} dz$ can be written as $\alpha+\beta\log(z)+R(z)$ for some $\alpha\in \mathbb C$, $\beta\in \Qbar$ and $R \in \Qbar((z^{1/2}))$ with no constant term. We set $c:=\exp(-\alpha)$ (which now completely defines $g(z)$) so that the local expansion at $z=0$ of $\widetilde{g}(z)$ is in $z^\beta\cdot \exp(P(z^{-1/2}))\cdot \Qbar[[z^{1/2}]]$ for some $P\in\Qbar[z]$ such that $P(0)=0$. 
We shall prove below that $\exp(-\int r(z) dz)$ is algebraic over $\Qbar(z)$ (up to a mutiplicative constant), from which we shall deduce that $\widetilde{g}(z)\in \textup{NGA}\{0\}_0^{\Qbar}$. We will then prove that $\widetilde{g}(z)$ is algebraic over $\Qbar(z)$.

Using $g'=(r+\sqrt{s})g$ and $g''=((r+\sqrt{s})'+(r+\sqrt{s})^2)g$, it is immediate to check that $g(z)$ is a solution of the operator
\begin{equation}\label{eq:opM}
M:=\frac{d^2}{dz^2}-\Big(2r+\frac{s'}{2s}\Big)\frac{d}{dz}+r^2-r'+\frac{rs'}{2s}-s \in \Qbar(z)\Big[\frac d{dz}\Big].
\end{equation}
(Another solution of $M$ is $h(z)$.) 
Let us perform the euclidean right-division of $L$ by $M$: there exist $P\in \Qbar(z)\setminus\{0\}$ and $R\in \Qbar(z)[\frac{d}{dz}]$ such that $L=PM+R$ where the order of $R$ is 0 or 1. Because $Lg(z)=Mg(z)=0$, we have $Rg(z)=0$. Comparing with $g'=(r+\sqrt{s})g$, we deduce that necessarily $R=0$. Consequently, $L=PM$ and thus $M$ is also a $G$-operator for these particular $r(z)$ and $s(z)$.

Therefore $M$ is fuchsian with rational exponents: denoting by $z_1, \ldots, z_m\in\Qbar$ its (pairwise distinct) finite singularities, we have
\begin{equation}\label{eq:rs}
-2r(z)-\frac{s'(z)}{2s(z)} = \sum_{j=1}^m \frac{1-\rho_{1,j}-\rho_{2,j}}{z-z_j}
\end{equation}
where for each $j$, $\rho_{1,j}\in \mathbb Q$ and $\rho_{2,j} \in \mathbb Q$ are the local exponents of $M$ at $z_j$ (see \cite[p.~77, Eqs.~(20) and~(21)]{poole}). Writing $s(z)=\delta\prod_{j=1}^k (z-w_j)^{s_j}$ where $\delta\in \Qbar^*$ and $w_j\in \Qbar$, $s_j\in \mathbb Z\setminus\{0\}$ for each $j$, we have 
\begin{equation}\label{eq:s}
\frac{s'(z)}{s(z)}=\sum_{j=1}^k \frac{s_j}{z-w_j}.
\end{equation}
Using \eqref{eq:s} in \eqref{eq:rs}, we deduce that 
\begin{equation}\label{eq:r}
r(z)=\sum_{j=1}^\ell \frac{t_j}{z-\alpha_j}
\end{equation}
for some $\alpha_j\in \Qbar$ and $t_j\in \mathbb Q$ for every $j$. From \eqref{eq:r}, it follows as claimed that there exists $d\in \mathbb C^*$ such that 
$$
\exp\Big(\int r(z) dz\Big) = d\prod_{j=1}^\ell (z-\alpha_j)^{t_j},
$$
ie that $\exp(\int r(z) dz)$ is algebraic over $\Qbar(z)$ up to a multiplicative constant. 
Hence, $\widetilde{g}(z)=d^{-1}\prod_{j=1}^\ell (z-\alpha_j)^{-t_j}g(z) \in \textup{NGA}\{0\}_0^{\mathbb{C}}$. Now we have also proved that the local expansion at $z=0$ of $\widetilde{g}(z)$ belongs to $z^\beta \cdot \exp(P(z^{-1/2}))\cdot \Qbar[[z^{1/2}]]$ for some $\beta\in\Qbar$ and $P\in\Qbar[z]$ such that $P(0)=0$; it follows that $\widetilde{g}(z) \in \textup{NGA}\{0\}_0^{\Qbar}$.

Now, the function $\widetilde{g}(z)$ is solution of the operator
$
N:=\frac{d^2}{dz^2}-\frac{s'(z)}{2s(z)}\frac{d}{dz}-s(z).
$
(This is formally the operator obtained from $M$ in \eqref{eq:opM} with $r(z)=0$.) Notice that $\widetilde{g}(z)$ is of order 2 over $\Qbar(z)$ because it is not 0 and $\widetilde{g}'=\sqrt{s}\widetilde{g}$ rules out the possibility that it is of order 1 over $\Qbar(z)$. Therefore, $N$ is a $G$-operator by Andr\'e's minimality theorem. 
A key remark is that $1/\widetilde{g}(z)$ is also a solution of $N$. 
Therefore, $1/\widetilde{g}(z)\in \textup{NGA}\{0\}_0^{\mathbb C}$. Since the local expansion at $z=0$ of $1/\widetilde{g}(z)$ 
belongs to $z^{-\beta} \cdot \exp(-P(z^{-1/2}))\cdot \Qbar[[z^{1/2}]]$ for some $\beta\in\Qbar$ and $P\in\Qbar[z]$ such that $P(0)=0$, 
 we deduce that $1/\widetilde{g}(z)$ is in $\textup{NGA}\{0\}_0^{\Qbar}$. By the scholie in \cite[p. 123]{andrelivre} (the proof of which encompasses our case), we conclude that $\widetilde{g}(z)$ is algebraic over $\Qbar(z)$. Therefore, up to a multiplicative constant in $\mathbb C$, $g(z)$ is also algebraic over $\Qbar(z)$.

The same arguments show that $h(z)$ is algebraic over $\Qbar(z)$, up to a multiplicative constant in $\mathbb C$. 
In other words, there exist $u,v\in \mathbb C^*$ such that $ug(z)$ and $vh(z)$ are algebraic over $\Qbar(z)$, and form a basis of $L$. This is in accordance with $(i)$ of Theorem \ref{theo:2}, the proof of which is now complete.

\section{Holonomic Nilsson-Gevrey arithmetic series}    \label{sec:holo}

Andr\'e introduced in  \cite{andre}
 the class of  Nilsson-Gevrey arithmetic series of order 0, defined as functions of the form 
 \begin{equation}\label{eqholo2}
\sum_{(\alpha,j,k)\in \mathbb S} \lambda_{\alpha,j,k} z^{\alpha} \log(z)^j f_{\alpha,j,k}(z)
\end{equation}
where  $\mathbb S$ is a finite subset of $\mathbb Q\times \mathbb N\times \mathbb N$, $\lambda_{\alpha,j,k}\in \mathbb C$ and each $f_{\alpha,j,k}(z)$ is a   power series 
$
\sum_{n=0}^{\infty} a_n z^n $ with algebraic coefficients $a_n$ (which depend also on $ \alpha$, $j$, $k$)
such that:
\begin{enumerate} 
\item[--] there exists $C>0$ such that 
for any $\sigma \in \textup{Gal}(\overline{\mathbb{Q}}/\mathbb Q)$, we have $\vert \sigma(a_n)\vert \le C^{n+1}$;
\item[--]  there exists a sequence of positive integers $d_n$ such that $d_n \le C^{n+1}$ and $d_n a_m$ is an algebraic integer for all $m\le n$.
\end{enumerate}

In other words, $f_{\alpha,j,k}(z)$ satisfies the requirements to be a $G$-function, except that it is not assumed to be holonomic. In this section we prove Proposition \ref{propholo} stated in the introduction, namely: a Nilsson-Gevrey arithmetic series of order 0 is holonomic if, and only if, it can be written as \eqref{eqholo2} with $G$-functions $f_{\alpha,j,k}(z)$. Indeed we shall prove the following more precise result.
 
\begin{prop}\label{propholo2}
Let $J\geq 0$, and $\calA\subset\C$ be a finite subset such that $\alpha-\alpha'\in\Z$ with $\alpha,\alpha'\in\calA$ implies $\alpha=\alpha'$. For any pair $(\alpha,j)\in\calA\times\{0,\ldots,J\}$ let $K(\alpha,j)$ be a non-negative integer,  $(f_{\alpha,j,k})_{1\leq k \leq K(\alpha,j)}$ be a family of $K(\alpha,j)$ functions holomorphic at 0 with algebraic Taylor coefficients, and  $(\lambda_{\alpha,j,k})_{1\leq k \leq K(\alpha,j)}$ be complex numbers linearly independent over $\Qbar$. Then the function
$$
f(z) = \sum_{ \alpha \in \calA} \sum_{j=0}^J \sum_{k=1}^{ K(\alpha,j)}  \lambda_{\alpha,j,k} z^{\alpha} \log(z)^j f_{\alpha,j,k}(z)
$$
is holonomic if, and only if, all functions $ f_{\alpha,j,k}(z)$ are holonomic.
\end{prop}

This result shows that  Proposition \ref{propholo} can be adapted easily to Nilsson-Gevrey arithmetic series of negative 
order (for instance with $E$-functions instead of $G$-functions). More generally, the assumptions on the growth and denominators of the Taylor coefficients of the power series $  f_{\alpha,j,k}(z)$ are not necessary here.

\bigskip

Since each  $f_{\alpha,j,k}(z)$ has algebraic Taylor coefficients at 0, it is holonomic (i.e., solution of a   differential equation with coefficients in $\C(z)$) if, and only if, it is solution of a   differential equation with coefficients in $\Qbar(z)$.  Proposition \ref{propholo} shows that the same property holds with $f(z)$. 

\bigskip

\begin{proof}
If all $ f_{\alpha,j,k}(z)$ are holonomic, then so is $f$. To prove the converse, we assume that $f$ is holonomic and for each $(\alpha,j)\in\calA\times\{0,\ldots,J\}$  we consider
 \begin{equation}\label{eqholop3}
 f_{\alpha,j}(z) =  \sum_{k=1}^{ K(\alpha,j)}  \lambda_{\alpha,j,k}  f_{\alpha,j,k}(z) \in\C[[z]]
\end{equation}
so that
 \begin{equation}\label{eqholop1}
 f(z) = \sum_{ \alpha \in \calA} \sum_{j=0}^J  z^{\alpha} \log(z)^j f_{\alpha,j}(z).
\end{equation}

Our first step is to prove that $ f_{\alpha,j}$ is holonomic for any pair  $(\alpha,j)$. With this aim in view, for each $\alpha \in\calA$ we denote by $J_\alpha$ the largest integer $j$ such that $f_{\alpha,j}$ is not identically zero; shrinking $\calA$ if necessary we assume that for each $\alpha$ there exists such a $j$. Then we shall prove by induction on $S := \sum_{\alpha\in\calA} (1+J_\alpha)$ that if $f$ is holonomic in Eq. \eqref{eqholop1} then all $f_{\alpha,j}$ are.

This property holds trivially if $S=0$, because $\calA = \emptyset$ in this case. Let us assume that $S>0$, and that it holds for $S-1$. Since $S>0$ we have $\calA \neq \emptyset$; we choose $\alpha_0\in\calA$. Upon dividing by $z^{\alpha_0}$ we may assume that $\alpha_0=0$. We denote by $T$ the monodromy around the origin. Then Eq. \eqref{eqholop1} provides an expression of the form
 \begin{equation}\label{eqholop2}
 (Tf-f)(z) = \sum_{ \alpha \in \calA' } \sum_{j=0}^{J'_\alpha}  z^{\alpha} \log(z)^j g_{\alpha,j}(z) 
\end{equation}
with functions $g_{\alpha,j}(z) $ holomorphic at 0, where  $\calA\setminus\{0\}\subset\calA'\subset\calA$ and 
 $J'_\alpha = J_\alpha$ for any $\alpha\in \calA\setminus\{0\}$. Moreover, if $J_0=0$ then $0\not\in\calA'$; otherwise $0\in\calA'$ and $J'_0=J_0-1$. In both cases Eq. \eqref{eqholop2} corresponds to $S' =  \sum_{\alpha\in\calA'} (1+J'_\alpha)=S-1$. Using the induction  hypothesis we deduce that all $ g_{\alpha,j}$ are holonomic. Now since the  $f_{\alpha,j}$ and $ g_{\alpha,j}$ are holomorphic at 0, it is not difficult to express all functions $ g_{\alpha',j'}$ as linear combinations of the  $f_{\alpha,j}$ (and $f_{0,0}$ does not appear in this computation since $T f_{0,0}-f_{0,0}=0$). The underlying matrix is invertible (since it is block-wise triangular with non-zero diagonal coefficients) so that any $f_{\alpha,j}$ (with  $(\alpha,j)\neq {0,0}$) is a $\C$-linear combination of the holonomic functions  $ g_{\alpha',j'}$, and therefore is holonomic. Using Eq. \eqref{eqholop1} it follows that $  f_{0,0}$ is holonomic too: this concludes the inductive proof.
 
 \bigskip
 
 Let us move now to the second part of the proof of Proposition \ref{propholo2}. Recall from Eq.  \eqref{eqholop3} that 
 \begin{equation}\label{eqholop4}
 f_{\alpha,j}(z) =  \sum_{k=1}^{ K(\alpha,j)}  \lambda_{\alpha,j,k}  f_{\alpha,j,k}(z).
\end{equation}
For simplicity we write $h(z) =  f_{\alpha,j}(z)$, $h_k(z) = f_{\alpha,j,k}(z)$, and $\lambda_k =   \lambda_{\alpha,j,k} $; then $h = \sum_{k=1}^K \lambda_k h_k$ with $K =  K(\alpha,j)$ is holonomic, the complex numbers $\lambda_k$ are $\Qbar$-linearly independent, and $h_k\in\Qbar[[z]]$. Our aim is to prove that all $h_k$ are holonomic.

We consider the subspace $F$ of $\C^K$ consisting in all $\asoul = (a_1,\ldots,a_K)$ such that $\sum_{k=1}^K a_k h_k$ is holonomic. We are going to prove that $F$ is defined over $\Qbar$, i.e. that there exists a basis of $F$ consisting of elements of $\Qbar^K$ (this is equivalent to the existence of a system of linear equations with algebraic coefficients that defines $F$: see \cite{Bourbaki}).
Since $(\lambda_1,\ldots,\lambda_K)\in F$ and the  $\lambda_k$ are $\Qbar$-linearly independent, this implies $F=\C^K$ so that all $h_k$ are holonomic.

Since $F$ is finite-dimensional, there exists a non-zero differential operator $L\in \C[z, \frac{d}{dz}]$ such that $L( \sum_{k=1}^K a_k h_k)=0$ for any $\asoul\in F$. Let $\mu$ and $\delta$ denote, respectively, the order and degree of $L$. We consider the subspace $V$ of $\C[[z]]$ spanned by the power series $z^i (  \frac{d}{dz} )^j h_k $ with $0\leq i \leq \delta$, $0\leq j \leq \mu$, $1 \leq k \leq K$. Since $h_k\in \Qbar[[z]]$ for any $k$, this subspace $V$ is defined over $\Qbar$: there exists a basis $(v_1,\ldots,v_N)$ of $V$ such that $v_\ell\in \Qbar[[z]]$ for any $\ell$. For any $\asoul\in \C^K$ we denote by $M(\asoul)\in M_{N, (\mu+1)(\delta+1)}(\C)$ the matrix whose columns are given by the coordinates of $ z^i (  \frac{d}{dz} )^j ( \sum_{k=1}^K a_k h_k) $ in the basis $(v_1,\ldots,v_N)$; then each coefficient of $M(\asoul)$ is a $\Qbar$-linear combination of $a_1$, \ldots, $a_K$ (with coefficients independent from $\asoul$). Now for $\asoul\in\C^K$ the following assertions are equivalent :
\begin{itemize}
\item[$\bullet$] $\asoul\in F$
\item[$\bullet$] $\sum_{k=1}^K a_k h_k$ is annihilated by a non-zero differential operator of degree at most $\delta$ and order at most $\mu$
\item[$\bullet$] the columns of $M(\asoul)$ are linearly dependent (over $\C$)
\item[$\bullet$] $M(\asoul)$ has rank less than $(\mu+1)(\delta+1)$
\item[$\bullet$] All minors of size $(\mu+1)(\delta+1)$ of $M(\asoul)$ are equal to 0.
\end{itemize}
Now each minor of  $M(\asoul)$ is a polynomial in $a_1$, \ldots, $a_K$ with algebraic coefficients. Therefore $F$ is the zero locus in $\C^M$ of a finite family of polynomials with algebraic coefficients. Since $F$  is also a vector subspace of $\C^M$, as such it is defined over $\Qbar$. This concludes the proof of Proposition \ref{propholo2}.
\end{proof}

\section{A result on algebraic functions} \label{sec:3}

While searching for a proof of Theorem~\ref{theo:2}, we proved a result of independent interest which is not used in the paper. We have not seen it in the literature. M. Singer could not find it either and he sent us his own proof based on differential Galois theory. Our approach presented below is different.

\begin{prop} \label{propalg}
 Let $f(z)$ be a function algebraic over $\C(z)$, solution of a differential equation $Ly(z)=0$ with $L\in \Qbar(z)[\frac{d}{dz}]\setminus\{0\}$. Then there exist complex numbers $\lambda_1,\ldots,\lambda_p$ and functions $f_1(z),\ldots,f_p(z)$ algebraic over $\Qbar(z)$ such that $Lf_i(z)=0$ for each $i$ and 
$$f(z)=\sum_{i=1}^p \lambda_i f_i(z).$$
\end{prop}

\begin{proof} Let us fix a simply connected cut plane which does not contain any singularity of $L$, and on which a determination of $\log(z)$ is fixed. We choose a non-singular point $z_0\in\Qbar$ in this cut plane. A solution of $Ly=0$ will be considered as a function on this cut plane, and 
 identified with its Taylor expansion at $z_0$. Notice that if such a solution has algebraic Taylor coefficients at $z_0$, then it is algebraic over $\C(z)$ if, and only if, it is algebraic over $\Qbar(z)$. 

Let $p$ denote the order of $L$, and $(f_1,\ldots,f_p)$ be a basis of solutions of $Ly=0$, with algebraic Taylor coefficients at $z_0$. Denote by $\calA$ the $\Qbar$-vector space of all functions algebraic over $\Qbar(z)$, and let 
$$V = \calA \cap \Span_{\Qbar} (f_1,\ldots,f_p) $$
denote the set of solutions of $L$ that are algebraic over $\Qbar(z)$ and have algebraic Taylor coefficients at $z_0$. Let $(g_1,\ldots,g_r)$ be a basis of this $\Qbar$-vector space, and $h_1,\ldots,h_q$ be such that $(g_1,\ldots,g_r, h_1,\ldots,h_q)$ is a $\Qbar$-basis of the space $ \Span_{\Qbar} (f_1,\ldots,f_p) $ of solutions of $Ly=0$ with algebraic Taylor coefficients at $z_0$. We have $0\leq r,q\leq p$ and $r+q=p$. If $q=0$ then $r=p$ and $ \Span_{\Qbar} (f_1,\ldots,f_p) = V\subset\calA$: all functions $f_1,\ldots,f_p$ are algebraic over $\Qbar(z)$, and Proposition \ref{propalg} is proved.

Therefore we may assume that $q\geq 1$. There exist $\kappa_1,\ldots,\kappa_p\in\C$ such that 
$$f = \kappa_1h_1+ \ldots \kappa_qh_q + \kappa_{q+1}g_1+ \ldots \kappa_pg_r.$$
Since $ g_1,\ldots,g_r$ and $f$ are algebraic over $\C(z)$, so is $ \kappa_1h_1+ \ldots \kappa_qh_q = f - \kappa_{q+1}g_1 - \ldots \kappa_pg_r$. Let $\cali$ denote the set of all $I\subset\unq$ for which there exist complex numbers $\mu_i$, $i\in I$, not all zero, such that $\sum_{i\in I} \mu_i h_i(z)$ is algebraic over $\C(z)$. If $\kappa_i = 0$ for any $i\in \unq $ then $f = \kappa_{q+1}g_1 + \ldots+ \kappa_pg_r$ and Proposition \ref{propalg} is proved. Otherwise we have $\unq\in\cali$ by taking $\mu_i=\kappa_i$ for any $i$, and we shall deduce a contradiction by considering an element $I$ of $\cali$ with minimal cardinality.

If $I = \{i\}$ for some $i$, then $h_i$ is algebraic over $\C(z)$; since $h_i$ has algebraic Taylor coefficients at $z_0$, it is algebraic over $\Qbar(z)$, so that $h_i\in V = \Span_{\Qbar} (g_1,\ldots,g_r) $: this contradicts the definition of $h_1,\ldots,h_q$.

Therefore $\Card I \geq 2$. Let us choose $i_0\in I$; by minimality of $I$ we have $\mu_{i_0}\neq 0$. Let us prove the $\mu_i/\mu_{i_0}$ is transcendental for at least one $i\in I$. Indeed, if all these numbers were algebraic, then $\mu_{i_0}^{-1} \sum_{i\in I} \mu_i h_i(z)$ would have algebraic Taylor coefficients at $z_0$ while being algebraic over $\C(z)$. Therefore it would be algebraic over $\Qbar(z)$, and belong to $V = \calA \cap \Span_{\Qbar} (f_1,\ldots,f_p) $. Since it is a non-zero element of $\Span_{\Qbar} (h_1,\ldots,h_q)$ this would contradict the definition of $ h_1,\ldots,h_q$.

We have proved that there exists $i_1\in I$ such that $\mu_{i_1}/\mu_{i_0}$ is transcendental. 
Moreover, there exists $P(z,X) = \sum_{k=0}^t P_k(z)X^k \in \C[z,X]\setminus \{0\}$ such that $P(z, \sum_{i\in I} \mu_i h_i(z))=0$. In other words, all Taylor coefficients at $z_0$ of this function of $z$ are zero. Notice that each such coefficient is a polynomial with algebraic coefficients in the $\mu_i$, $i\in I$, and the coefficients of the $P_k(z)$. Let $\K$ denote a sub-field of $\C$, of finite transcendence degree over $\Qbar$, that contains the $\mu_i$, $i\in I$, and the coefficients of the $P_k(z)$. Denote by $d$ the transcendence degree of $\K$ over $\Qbar$. Since $\mu_{i_1}/\mu_{i_0}$ is transcendental and belongs to $\K$, we have $d\geq 1$. Let $\theta_1,\ldots, \theta_d$ denote a transcendence basis of $\K$ over $\Qbar$ (see for instance \cite[p.~109, Definition 9.8]{Milne}). Since $\C$ is not coutable, there exist $\alpha_1,\ldots, \alpha_d\in \C$ such that 
$\theta_1,\ldots, \theta_d, \alpha_1,\ldots, \alpha_d$ are algebraically independent over $\Qbar$. Denote by $\sigma : \Qbar(\theta_1,\ldots, \theta_d) \to \Qbar(\alpha_1,\ldots, \alpha_d)$ the morphism of $\Qbar$-algebras defined by $\sigma(\theta_j)=\alpha_j$ for any $j\in \{1,\ldots, d \}$; in other words, we have $\sigma(R( \theta_1,\ldots, \theta_d) ) = R( \alpha_1,\ldots, \alpha_d)$ for any $R\in \Qbar(X_1,\ldots,X_d)$. Since $\K$ is an algebraic extension of $\Qbar( \theta_1,\ldots, \theta_d)$, $\sigma$ can be extended to a morphism $\K\to\LL$ still denoted by $\sigma$, where $\LL$ is an algebraic extension of $\Qbar( \alpha_1,\ldots, \alpha_d)$ 
(see \cite[Proposition 2.2 and Theorem 6.8]{Milne}).
If $\sigma(\mu_{i_1}/\mu_{i_0}) = \mu_{i_1}/\mu_{i_0}$, then this element belongs to both $\K$ and $\LL$, so that it is algebraic over both $\Qbar(\theta_1,\ldots, \theta_d) $ and $ \Qbar(\alpha_1,\ldots, \alpha_d)$. Since $\theta_1,\ldots, \theta_d, \alpha_1,\ldots, \alpha_d$ are algebraically independent over $\Qbar$, this implies $\mu_{i_1}/\mu_{i_0}\in\Qbar$, which is a contradiction.
Therefore we have $\sigma(\mu_{i_1}/\mu_{i_0}) \neq \mu_{i_1}/\mu_{i_0}$. 

Now recall that all Taylor coefficients at $z_0$ of $P(z, \sum_{i\in I} \mu_i h_i(z))$ are zero, and that $P(z,X)\in \K[z,X]\setminus \{0\}$ by construction of $\K$. Denote by $P^\sigma(z,X)\in \LL[z,X]\setminus \{0\}$ the polynomial obtained from $P(z,X)$ by applying $\sigma$ to all coefficients. Recall that all Taylor coefficients at $z_0$ of the functions $h_i$ are algebraic, and accordingly invariant under $\sigma$.
Therefore any Taylor coefficient at $z_0$ of $P^\sigma(z, \sum_{i\in I} \sigma(\mu_i) h_i(z))$ is the image under $\sigma$ of the corresponding coefficient of $P(z, \sum_{i\in I} \mu_i h_i(z))$, which is zero. In other words, we have $P^\sigma(z, \sum_{i\in I} \sigma(\mu_i) h_i(z))=0$: the function $ \sum_{i\in I} \sigma(\mu_i) h_i(z)$ is algebraic over $\C(z)$. Therefore the function
$$h(z) = \frac1{\mu_{i_0}} \sum_{i\in I} \mu_i h_i(z) - \frac1{\sigma(\mu_{i_0})} \sum_{i\in I} \sigma(\mu_i) h_i(z)$$
is also algebraic over $\C(z)$. It is a linear combination of the $h_i$, $i\in I \setminus \{i_0\}$: by minimality of $I$, it has to be 0. Since $h_1,\ldots,h_q$ are linearly independent over $\C$ (because they are over $\Qbar$ and they have algebraic Taylor coefficients at $z_0$), we deduce that all coefficients are 0, so that $\frac{\mu_{i_1}}{\mu_{i_0}} = \frac{\sigma(\mu_{i_1})} {\sigma(\mu_{i_0})}$ which is a contradiction. 

This concludes the proof of Proposition \ref{propalg}.
\end{proof}

\def\refname{Bibliography}


\begin{thebibliography}{1}\label{sec:biblio}

\bibitem{andrelivre} Y. Andr\'e, $G$-functions and Geometry, Aspects of Math. {\bf E13}, Vieweg, Braunschweig/
Wiesbaden (1989).
\bibitem{andre} Y. Andr\'e, {\em S\'eries Gevrey de type arithm\'etique $I$.
Th\'eor\`emes de puret\'e et de dualit\'e}, Ann. of Math. {\bf 151} (2000), 705--740.
\bibitem{bombieri} E. Bombieri, {\em On $G$-functions}, in Recent progress in analytic number theory, Symp.
Durham 1979, Vol. 2, 1--67 (1981).
\bibitem{Bourbaki} N. Bourbaki, Algèbre, Chapitre II, 3rd ed., Hermann (1962).
\bibitem{bouw} I. Bouw, M. M\"oller, {\em Differential equations associated with nonarithmetic Fuchsian
groups}, J. Lond. Math. Soc. {\bf 81}.1 (2010), 65--90.
\bibitem{chud} G. V. Chudnovsky, {\em On applications of diophantine approximations}, Proc. Natl. Acad.
Sci. USA {\bf 81} (1984), 7261--7265.
\bibitem{dwork} B. Dwork, {\em Differential operators with nilpotent $p$-curvatures}, Amer. J. Math. {\bf 112}.5
(1990), 749--786.
\bibitem{firipadoue} S. Fischler, T. Rivoal, {\em On Siegel's problem for $E$-functions}, Rendiconti Sem. Mat. Padova, to appear, 27 pages.
\bibitem{galochkin} A. I. Galochkin, {\em Lower bounds of polynomials in the values of a certain class of analytic
functions}, Mat. Sb. {\bf 95} (137) (1974), 385--407.
\bibitem{gkz} I. M. Gel'fand, M. M. Kapranov, A. V. Zelevinsky, {\em Generalized Euler integrals and $A$-hypergeometric
functions}, Adv. Math. {\bf 84} (1990) 255--271.
\bibitem{gorelov1} V. A. Gorelov, {\em On the algebraic independence of values of E-functions at singular points and the Siegel conjecture}, Mat. Notes {\bf 67}.2 (2000), 174--190.
\bibitem{gorelov2} V. A. Gorelov, {\em 
On the Siegel conjecture for second-order homogeneous linear differential equations}, 
Math. Notes {\bf 75}.4 (2004), 513--529. 
\bibitem{katz87} N. M. Katz, {\em On the calculation of some differential galois groups}, Invent. math. {\bf 87} (1987), 13-- 61.
\bibitem{kovacic} J.~J.~Kovacic, {\em An algorithm for solving second order linear homogeneous differential equations}, J. Symbolic Computation {\bf 2} (1986), 3--43.
\bibitem{krammer} D. Krammer, {\em An example of an arithmetic Fuchsian group}, J. reine angew. Math. {\bf 473} 
(1996), 69--85.
\bibitem{lepetit} G. Lepetit, {\em Quantitative problems on the size of $G$-operators}, preprint (2020), 20 pages.
\bibitem{Milne} J. S. Milne, {\em Fields and Galois Theory}, Version 4.61, 138 pages, 
 available at www.jmilne.org/math/, 2020. 
\bibitem{poole} E. G. C. Poole, {\em Introduction to the theory of linear differential equations}, Dover Publications, New York, 1960.
\bibitem{rivroq} T. Rivoal, J. Roques, {\em On the algebraic dependence of $E$-functions}, Bull. London Math. Soc. {\bf 48}.2 (2016), 271--279.
\bibitem{rivroq2} T. Rivoal, J. Roques, {\em Siegel's problem for $E$-functions of order 2}, preprint (2016), 14 pages, to appear in the proceedings of the conference Transient Transcendence in Transylvania, Bra\c{s}ov, may 2019.
\bibitem{ulmerweil} F. Ulmer, J.-A. Weil, {\em Note on Kovacic's algorithm}, J. Symb. Comp. {\bf 22} (1996), 179--200.
\bibitem{siegel} C. L. Siegel, {\em \"Uber einige Anwendungen Diophantischer Approximationen}, Abh. Preuss. Akad.Wiss., Phys.-Math. Kl. (1929-30), no. 1, 1--70.
\bibitem{singer} M. F. Singer, {\em Algebraic Solutions of nth Order Linear Differential Equations}, Proceedings of the Queen's University 1979 Conference on Number Theory, Queens Papers in Pure and Applied Mathematics {\bf 54}, 379--420.
\bibitem{shid} A. B. Shidlovskii, {\em Transcendental Numbers}, W. de Gruyter Studies in Mathematics {\bf 12}, 1989.
\bibitem{sturmfels} B. Sturmfels, {\em Solving algebraic equations in terms of $A$-hypergeometric series}, Discrete Math. {\bf 210} (2000), 171--181.
\bibitem{VDPSinger} M. van der Put, M. Singer, {\em Galois Theory of Linear Differential Equations}, Grundlehren der mathematischen Wissenschaften {\bf 328}, Springer, 2003.
\bibitem{Zannier} U. Zannier, {\em On periodic mod $p$ sequences and $G$-functions (On a conjecture of Ruzsa)}, Manuscripta Math. {\bf 90} (1996), 391--402.

\end{thebibliography}
\end{document}